\documentclass[12pt,a4paper]{scrartcl}

\usepackage[T1]{fontenc} 
\usepackage{textcomp}    

\usepackage{setspace} 
\usepackage[leqno]{amsmath}
\usepackage{amsfonts,amssymb,amsthm}

\usepackage[all]{xy}
\SelectTips{cm}{12}   

\usepackage{relsize}  

\newcommand{\Cone}{\mathrm{cone}}

\newcounter{lemma}
\theoremstyle{plain}                        
\newtheorem{localtheorem}[lemma]{Theorem}
\newtheorem{locallemma}[lemma]{Lemma}
\newtheorem{localproposition}[lemma]{Proposition}
\newtheorem{localcorollary}[lemma]{Corollary}
\newtheorem{localconstruction}[lemma]{Construction}

\theoremstyle{definition}
\newtheorem{localremarknumbered}[lemma]{Remark}
\newtheorem*{localremark}{Remark}
\newtheorem*{localdefinition}{Definition}

\newtheorem*{localremarks}{Remarks}

\newcommand{\Coh}{\mathrm{Coh}}
\newcommand{\Db}{\mathrm{D^b}}

\newcommand{\FM}{\mathsf{FM}}

\DeclareMathOperator{\Pic}{Pic}

\DeclareMathOperator{\Hilb}{Hilb}

\newcommand{\Sym}{\mathrm{Sym}}

\DeclareMathOperator{\Ext}{Ext}
\DeclareMathOperator{\Hom}{Hom}

\newcommand{\shExt}{\mathcal{E}xt}

\DeclareMathOperator{\ch}{ch}

\DeclareMathOperator{\rk}{rk}

\newcommand{\dual}{^\vee}
\newcommand{\ddual}{^{\vee\vee}}
\newcommand{\inv}{^{-1}}
\newcommand{\orth}{^\perp}
\newcommand{\rarpa}[1]{\stackrel{#1}{\rightarrow}}
\newcommand{\larpa}[1]{\stackrel{#1}{\longrightarrow}}
\newcommand{\isom}{ \text{{\hspace{0.48em}\raisebox{0.8ex}{${\scriptscriptstyle\sim}$}}}
                    \hspace{-0.65em}{\rightarrow}\hspace{0.3em}} 

\newcommand{\surject}{\rightarrow\hspace{-1.8ex}\rightarrow}

\newcommand{\kimplies}{$\Longrightarrow\,\,$}

\newcommand{\IC}{\mathbb C}

\newcommand{\IN}{\mathbb N}

\newcommand{\IP}{\mathbb P}

\newcommand{\IZ}{\mathbb Z}

\newcommand{\kc}{\mathcal{C}}

\newcommand{\kj}{\mathcal{J}}

\newcommand{\km}{\mathcal{M}}

\newcommand{\ko}{\mathcal{O}}
\newcommand{\kp}{\mathcal{P}}

\newcommand{\kt}{\mathcal{T}}

\renewcommand{\tilde}[1]{\widetilde{#1}}

\newcommand{\coker}{\mathrm{coker}}
\newcommand{\length}{\mathrm{length}}
\newcommand{\im}{\mathrm{im}}
\newcommand{\supp}{\mathrm{supp}}

\newcommand{\pr}{\mathrm{pr}}
\newcommand{\chara}{\mathrm{char}}

\newcommand{\proofpara}[1]{\medskip\noindent\emph{#1}}

\newcommand{\map}[1]{\stackrel{#1}{\longrightarrow}}

\newcommand{\bib}[4]{\bibitem{#1} #2: \emph{#3}, #4.}

\newcounter{abccounter}
\newenvironment{abcliste}{\begin{list}{(\alph{abccounter})}
                      {\usecounter{abccounter}
                       \setlength{\topsep}{0ex}
                       \setlength{\partopsep}{0ex}
                       \setlength{\listparindent}{0ex}
                       \setlength{\itemsep}{0ex}
                       \setlength{\parsep}{0ex}
                       \setlength{\leftmargin}{3em}
                       \setlength{\labelwidth}{2em}
                       \setlength{\parskip}{0ex}
                      }
                      }{\end{list}}

\begin{document}

\begin{center}
\textbf{\Large Postnikov-Stability for Complexes}

\bigskip

Georg Hein\footnote{
Universit\"at Duisburg-Essen, FB Mathematik D-45117 Essen,
\texttt{georg.hein@uni-due.de}}
, David Ploog\footnote{
Freie Universit\"at Berlin, FB Mathematik, Arnimallee 3, 
D-14195 Berlin, \texttt{ploog@math.fu-berlin.de}}
\\
April 19, 2007
\end{center}

\begin{quote}{\small\scshape Abstract}
We present a novel notion of stable objects in the
derived category of coherent sheaves on a smooth projective variety.
As one application we compactify a moduli space of stable bundles 
using genuine complexes.
\end{quote}

\subsection*{Introduction}

Let $X$ be a polarised, smooth projective variety of dimension $n$ over an
algebraically closed field $k$. Our aim is to introduce a stability notion
for complexes, i.e.\ for objects of $\Db(X)$, the bounded derived category
of coherent sheaves on $X$. The main motivation for this notion is Falting's
observation that semistability on curves can be phrased as the existence of
non-trivial orthogonal sheaves \cite{Faltings} (similar results hold for surfaces,
see \cite{hein}). In order to make this idea work, we need convolutions and
Postnikov systems (the former can be seen as a generalisation of total 
complexes, and the latter generalises filtrations to the derived category). 
The details will be spelt out in the next section. As an example
of our theory, we show how a classical non-complete moduli space of certain
bundles can be compactified using complexes (see Section \ref{compactify}).
Also, by construction, our notion of stability is preserved under equivalences
(Fourier-Mukai transforms). Of interest to us is when classical
preservation of stability conditions is a special case of our situation. A
first check is done in Section \ref{conservation}.
In Sections \ref{sheafconditions} and \ref{lemmasection}, we give some general
facts of projective geometry from the derived point of view. In particular,
Lemma \ref{lemma-2}, a generalisation of the Euler sequence, is used several
times.

It seems only fair to point out that the results of this article in all
probability bear no connection with Bridgeland's notion of t-stability on 
triangulated categories (see \cite{bridgeland}). His starting point about
(semi)stability in the classical setting is the Harder-Narashiman filtration
whereas, as mentioned above, we are interested in the possibility to capture 
$\mu$-semistability in terms of Hom's in the derived category. Our approach 
is much closer to, but completely independent of, Inaba (see \cite{inaba}).

On notation: we will denote the $i$-th homology of a complex $a$ 
by $h^i(a)$. Functors are always derived without additional notation; e.g.\
for a proper map $f:X\to Y$ of schemes, we write $f_*:\Db(X)\to\Db(Y)$ for the
triangulated (exact) functor obtained by deriving $f_*:\Coh(X)\to\Coh(Y)$. 
For two objects $a$, $b$ of a $k$-linear triangulated category, we write
$\Hom^i(a,b):=\Hom(a,b[i])$ and $\hom^i(a,b):=\dim_k\Hom^i(a,b)$.
The Hilbert polynomial of a sheaf $E$ is denoted by $p(E)$, so that 
$p(E)(l)=\chi(E(l))$. Finally, by semistability for sheaves, we always mean 
$\mu$-semistability.

\subsection*{P-stability}

Let $\kt$ be a $k$-linear triangulated category for some field $k$;
we usually think of $\kt=\Db(X)$, the bounded derived category of a smooth,
projective variety $X$, defined over an algebraically closed field $k$.
A \emph{Postnikov-datum} or just \emph{P-datum} is a finite collection
$C_d, C_{d-1},\dots, C_{e+1}, C_e\in\kt$ of objects together with
nonnegative integers $N_i^j$ (for $i,j\in\IZ$) of which only a finite number
are nonzero. We will write $(C_\bullet,N)$ for this.

Recall the notions of Postnikov system and convolution (see 
\cite{GM}, \cite{BBD}, \cite{Orlov}, \cite{Kawamata}):
given finitely many objects $A_i$ (suppose $n\geq i\geq 0$) of $\kt$
together with morphisms $d_i:A_{i+1}\to A_{i}$ such that $d^2=0$,
a diagram of the form
\[ \xymatrix@=1em{
  A_n    \ar[rr]^{d_{n-1}} \ar@{=}[ddr] & &
  A_{n-1} \ar[rr]^{d_{n-2}} \ar[ddr]^{i} & &
  A_{n-2} \ar[r]           \ar[ddr]^{i} &
  \cdots\cdots & & \cdots\cdots\ar[r] & 
  A_1 \ar[rr]^{d_{0}}      \ar[ddr]^{i} & &
  A_0                     \ar[ddr] \\ \\
&  T_n    \ar[uur]^{j} & & 
  T_{n-1} \ar[uur]^{j}  \ar[ll]^{[1]} & & 
  T_{n-2}               \ar[ll]^{[1]} &
  \cdots\cdots \ar[l] & 
  T_2    \ar[uur]      \ar[l] & &
  T_{1}   \ar[uur]      \ar[ll]^{[1]} & &   
  T_{0}                 \ar[ll]^{[1]} 
} \]
(where the upper triangles are commutative and the lower ones are distinguished)
is called a \emph{Postnikov system} subordinated to the $A_i$ and $d_i$. The
object $T_0$ is called the \emph{convolution} of the Postnikov system.

\begin{localdefinition}
An object $A\in\kt$ is \emph{P-stable with respect to $(C_\bullet,N)$} if
\begin{abcliste}
\item[(i)]
      $\hom^j_{\kt}(A,C_i)=N_i^j$ for all $i=d,\dots,e$ and all $j$.
\item[(ii)]
      For $i>0$, there are morphisms $d_i:C_i\to C_{i-1}$ such that $d^2=0$ and
      that the complex $(C_{\bullet\ge0},d_\bullet)$ admits a convolution $K$.
\item[(iii)]
      $\Hom_{\kt}^*(A,K)=0$, i.e.\ $K\in A\orth$.
\end{abcliste}
\end{localdefinition}


%

\begin{localremark}
$\text{\phantom{xxx}}$
\begin{abcliste}
\item Convolutions in general do not exist, and if they do, there is no
      uniqueness in general, either. There are restrictions on the
      $\Hom^j(C_a,C_b)$'s which ensure the existence of a (unique) convolution.
      For example, if $\kt=\Db(X)$ and all $C_i$ are sheaves, then the unique
      convolution is just the complex $C_\bullet$ considered as an object of $\Db(X)$.
%
%
%
\item Note that the objects $C_i$ with $i<0$ do not take part in forming the
      Postnikov system. We call the conditions enforced by these objects
      via (i) the \emph{passive} stability conditions. They can be used to
      ensure numerical constraints, like fixing the Hilbert polynomial of
      sheaves.
\item In many situations there will be trivial choices that ensure P-stability.
      This should be considered as a defect of the parameters (like choosing
      non-ample line bundles when defining $\mu$-stability) and not as a defect
      of the definition.
\end{abcliste}
\end{localremark}


\section{Example: Stability on algebraic curves}
In this section, $X$ denotes a smooth projective curve of genus $g$
over $k$.
Let $r > 0$ and $d$ be two integers and fix a line bundle $L_1$ on $X$ 
of degree one.

\subsection{Semistability conditions on curves}
Our starting point is the following result. The vector bundle 
$F_{r,d}$ appearing in statement (iii) of the theorem below is 
universal, i.e.\ it only depends on $r$, $d$, and $L_1$. It is
constructed in Section \ref{lemmasection} on page \pageref{Sm}. 
The specifications for the construction are given in the proof
below.

\begin{localtheorem}\label{cur-stab-con}
For a coherent sheaf $E$ on $X$ of rank $r$ and degree $d$, the
following conditions are equivalent:

\begin{tabular}{lp{13cm}}
(i)   & $E$ is a semistable vector bundle.\\
(ii)  & There is a sheaf $0\ne F\in (E\dual)\orth$, i.e.\
        $H^0(E \otimes F)= H^1(E \otimes F) =0$. \\
(ii') & There exists a vector bundle $F$ on $X$ with
        $\det(F) \cong L_1^{\otimes (r^2(g-1)-rd)}$ and
	$\rk(F)=r^2$ such that
        $H^0(E \otimes F)= H^1(E \otimes F) =0$.\\
(iii) & $\Hom(E,F_{r,d})=0$.\\
\end{tabular}
\end{localtheorem}

\begin{proof}
(ii') \kimplies (ii) is trivial and (ii) \kimplies (i) is well-known.
The implication (i) \kimplies (ii') was shown in Popa's paper \cite{Popa}.
Thus, it suffices to show (ii') \kimplies (iii) \kimplies (i).

Suppose there exists such a vector bundle $F$ as in (ii'). It follows,
that $F$ is also a semistable vector bundle. 
Putting $e:=g+1+\lceil\frac{d}{r}\rceil$, the bundle
$F\otimes L_1^{\otimes e}$ is then globally generated. Since $X$ is of
dimension one there exists a surjection
 $\xymatrix{
     \ko_X ^{\oplus (r^2+1) } \ar@{->>}[r]^-\pi &
     F \otimes L_1^{\otimes e} }$.
Its kernel is the line bundle
 $\ker(\pi) \cong
  \det\left(F \otimes L_1^{\otimes e}\right)\inv \cong
  L_1^{\otimes e'}$
with $e':=rd-2gr^2-r^2\lceil \frac{d}{r}\rceil$.
Eventually, we obtain a short exact sequence
\[\xymatrix@1{ 0 \ar[r] & A \ar[r]^{\alpha} & B \ar[r] & F \ar[r] & 0 } \]
\[ \mbox{ with }
  A=L_1^{\otimes(rd-2gr^2-(r^2+1)\lceil \frac{d}{r}\rceil -g-1)} 
   \mbox{ and }
  B=\left( L_1^{\otimes (-g-1-\lceil \frac{d}{r}\rceil)} 
     \right)^{\oplus (r^2+1)}.\]
The semistability of $E$ implies that
 $H^0(E \otimes A)=H^0(E \otimes B)=0$.
Thus, the existence of a vector bundle $F$ with the above properties
is equivalent to the existence of a morphism
$\alpha \in \Hom(A,B)$
such that the resulting homomorphism $H^1(E \otimes A) \to
H^1(E \otimes B)$ is injective. Invoking Remark \ref{cor-2},
this is equivalent to (iii). 

Suppose (iii) holds. If $E\surject E''$ were a destabilising quotient,
then we had $\mu(E'') < \mu(E)-\frac{1}{r^2}$. Since
 $\mu(F_{r,d}) > \mu(E'')-(g-1)$ we conclude $\Hom(E'',F_{r,d}) \ne 0$
which contradicts (iii).
\end{proof}

\subsection{A first P-stability datum for algebraic curves}
We consider the derived category $\Db(X)$ of the smooth projective curve
$X$. Let $L$ be a very ample line bundle of degree $D$ on $X$.
As before, we fix two integers $r>0$ and $d$. We assume that
$d>(2g-2+D)r$. We have the

\begin{localproposition}\label{curPdatum1}
For an object $e\in\Db(X)$ the following conditions are equivalent:
\begin{abcliste}
\item[(i)]  $e$ is a semistable sheaf of rank $r$ and degree $d$.
\item[(ii)] $e$ satisfies the following numerical conditions, for all $i\neq0$:
\begin{align*}
  &\hom(\ko_X,e) = d-r(g-1),                & & \hom(\ko_X,e[i]) = 0, \\
  &\hom(L,e) = d-r(g-1-D),                  & & \hom(L,e[i]) = 0, \\
  &\hom(L^{\otimes (r(g-1-D)-d)},e[i]) = 0, & & \hom(e,F_{r,d}) = 0.
\end{align*}
\end{abcliste}
\end{localproposition}

\begin{proof}
If $e$ is a sheaf as in (i), then $\hom(e,F_{r,d})=0$ follows from 
Theorem \ref{cur-stab-con}; for the other equations in (ii), we use
$\chi(E(k))=rDk+d-r(g-1)$ and note the vanishing
$H^1(E)=H^1(E\otimes L\inv)=0$ due to semistability and the assumption
$d>(2g-2+D)r$.

To see that (ii) implies (i), we use again that $\hom(e,F_{r,d})=0$
entails the stability of $e$, as  by Lemma \ref{sc-1} the other five
identies grant in advance that $e$ is a sheaf of rank $r$ and degree $d$.
\end{proof}

\begin{localremark}
The conditions in part (ii) of Proposition \ref{curPdatum1} give a
P-stability datum with stable objects the semistable vector bundles of
rank $r$ and degree $d$. Note that only passive stability conditions
take part. Using Serre duality, the first five conditions easily brought
in the form $\hom^j(e,C_i)=N^j_i$ demanded in the definition.
\end{localremark}

\begin{localremark}
The above condition $d>(2g-2+D)r$
on the degree of our semistable vector bundles is no
restriction. By twisting the vector bundles with a line bundle of
sufficiently high degree this condition is always satisfied.
\end{localremark}

\subsection{Another P-stability datum for algebraic curves}
As before, we consider the derived category $\Db(X)$
of a smooth projective curve $X$.
Fix integers $r$ and $d$.
We consider the two vector bundles
\[A=L_1^{\otimes(rd-2gr^2-(r^2+1)\lceil \frac{d}{r}\rceil -g-1)}
\mbox{ and }
B=\left( L_1^{\otimes (-g-1-\lceil \frac{d}{r}\rceil)} \right)^{\oplus
(r^2+1)} \]
from the proof of Theorem \ref{cur-stab-con}.

\begin{localproposition}\label{curPdatum2}
For an object $e \in \Db(X)$ the following conditions are equivalent:
\begin{abcliste}
\item[(i)]  $e$ is a semistable sheaf of rank $r$ and degree $-d$.
\item[(ii)] There exists a morphism $\xymatrix{A \ar[r]^\psi & B}$ 
            such that $\Hom(e,\Cone(\psi)[i])=0$ for all $i\in\IZ$ and 
            $e$ satisfies the following conditions
\begin{align*}
\hom(e,A[1]) &= (2g+ \lceil\genfrac{}{}{}{1}{d}{r}\rceil - \genfrac{}{}{}{1}{d}{r})(r^3+r), &
\hom(e,A[i]) &= 0  \text{ for } i \ne 1, \\
\hom(e,B[1]) &= (2g+ \lceil\genfrac{}{}{}{1}{d}{r}\rceil- \genfrac{}{}{}{1}{d}{r})(r^3+r), &
\hom(e,B[i]) &= 0 \text{ for } i \ne 1.
\end{align*}
\end{abcliste}
\end{localproposition}

\begin{proof}
If $e \in \Db(X)$ is a semistable vector bundle of rank $r$ and degree
$-d$, then the slope of $e$ is bigger than the slope of the semistable
vector bundles $A$ and $B$. Hence, we have $\Hom(e,A)=\Hom(e,B)=0$, and
we can compute the dimensions of $\Hom(e,A[1])$ and $\Hom(e,B[1])$ using
the Riemann-Roch theorem. The existence of a map $\psi \in \Hom(A,B)$
with the property $\Hom(e,\Cone(\psi))=0$ is a consequence of the proof
of the implication (ii') $\implies$ (iii) in Theorem \ref{cur-stab-con}.
Since both $e$ and $\Cone(\psi)$ are sheaves and we also have 
$\chi(e,\Cone(\psi))=0$, the vanishing $\Hom^*(e,\Cone(\psi))=0$ follows.

Now suppose that $e \in \Db(X)$ fulfills condition (ii). It follows that
$\psi \in \Hom(A,B)$ is not trivial. Since $A$ is a line bundle $\psi$
must be injective. Thus, the cone of $\psi$ is just the cokernel $F$ of
$\psi:A \to B$. Since $X$ is a smooth curve, $e$ is isomorphic to
its cohomology, that is $e= \oplus_{i \in \IZ}e_i[-i]$ with all $e_i$ 
coherent sheaves.
Since $e$ is orthogonal to $F$, all the $e_i$ are orthogonal to $F$ and
are semistable vector bundles of slope $\frac{-d}{r}$ by Theorem
\ref{cur-stab-con}. However, as in the proof of (i) $\implies$
(ii), $e_i \ne 0$ forces $\hom(e_i,A[1])$ to be positive. So we
eventually conclude $e_i =0$ for all $i \ne 0$.
\end{proof}

\begin{localremark}
The condition (ii) in Proposition \ref{curPdatum2}
gives a second P-stability datum
on an algebraic curve. Here we have the additional feature that any
stable object $e$ defines a divisor $\Theta_e$ by
\[ \Theta_e = \left\{ \psi \in \Hom(A,B) \:|\: \Hom(e,\Cone(\psi)) \ne 0
\right\} .\] This divisor is invariant under the standard $k^*$-action
on $\Hom(A,B)$. Thus, we obtain the $\Theta$-divisor
$\Theta_e\subset\IP(\Hom(A,B)\dual)$. A straightforward computation shows that
$\deg(\Theta_e)=(2g+ \lceil\frac{d}{r} \rceil-\frac{d}{r})(r^3+r)$.

The assignment $e \mapsto \Theta_e$ allows an identification of stable 
objects with points in some projective space (namely the linear system of the
$\Theta$-divisors).
P-equivalence of stable objects can be defined by $e \sim_P e'$,
if and only if $\Theta_e = \Theta_{e'}$.
It turns out that in this case P-equivalence classes
coincide with S-equivalence classes.
\end{localremark}

\subsection{Preservation of semistability on an elliptic curve} \label{conservation}

\paragraph{The moduli space of torsion sheaves of length $r$} $ $\\
Now let $X$ be an elliptic curve with a fixed point $P \in X(k)$
and fix a positive integer $r$.
We propose to consider semistable vector bundles of rank
$r$ and degree zero on $X$.

In order to do so, we first consider the following P-stability datum:
an object $t \in \Db(X)$ is P-stable, if and only if there exists a
morphism $\alpha:\ko_X(-3P) \to \ko_X$ with $\hom(t[j],\Cone(\alpha))=0$
for all $j\in\IZ$ and such that for all $i\neq0$ holds
\begin{align*}
 & \hom(\ko_X(-3P),t)    = r, & &\hom(\ko_X,t)    = r, \\
 & \hom(\ko_X(-3P),t[i]) = 0, & &\hom(\ko_X,t[i]) = 0.
\end{align*}
Obviously, we have $t \in \Db(X)$ is P-stable $\Leftrightarrow$ $t$ is a
torsion sheaf of length $r$. The $\Theta$-divisor associated to $t$ is
 $\Theta_t := \{ \alpha\in\IP(H^0(\ko_X(3P))^\vee)
                  \:|\: \Hom(t,\Cone(\alpha))\ne0\}$.
It is a union of $r$ lines, one for every point in the support of $t$ 
(counted with multiplicities).
Two torsion sheaves $t$ and $t'$ are P-equivalent if their
$\Theta$-divisors $\Theta_t$ and $\Theta_{t'}$ coincide.
If the $\Theta$-divisor is reduced, the P-equivalence class contains
only isomorphic objects. However, the maximal number of isomorphism
classes in a P-equivalence class is the number of partitions of $r$.
Grothendieck's Hilbert scheme $\Hilb^r(X)$ (see \cite{Gro}) of length $r$
torsion quotients of $\ko_X$ is the moduli space parametrising the
equivalence classes.

\paragraph{The Fourier-Mukai transform associated to the Poincar\'e bundle} $ $\\
We consider the product $X \times X$ with projections $\pr_1$ and
$\pr_2$.
Let $\Delta\subset X \times X$ be the diagonal, and
 $\kp:= \ko_{X \times X}(\Delta) \otimes \pr_1^*\ko_X(-P)\otimes \pr_2^*\ko_X(-P)$
the Poincar\'e line bundle. We consider the Fourier-Mukai transform
(recall that $\pr_{2*}$ is the derived push-forward)
\[ \FM_\kp:\Db(X) \to \Db(X) \qquad
    t \mapsto \pr_{2*}(\kp \otimes \pr_1^*t). \]
We set $M_1:=\FM_\kp(\ko_X(-3P))$, and $M_0:=\FM_\kp(\ko_X)$.
The complex $M_1$ is a sheaf shifted by $[1]$, with $M_1[-1]$ being 
locally free and $\rk(M_1[-1])=3$ and $\deg(M_1[-1])=1$. The complex $M_0$ 
is a shifted skyscraper sheaf: $M_0[-1]=k(P)$.

\paragraph{A P-stability datum for rank $r$ bundles of degree zero} $ $\\
Let $t \in \Db(X)$ be P-stable with respect to the above P-datum. Then
$e:=\FM_\kp(t)$ is P-stable with respect to the following P-datum:
\begin{align*}
&\hom(M_j,e)=r,\qquad \hom(M_j,e[i])=0 \quad \text{ for } j\in \{0,1\},\ i \ne 0, \\
&\text{and there is } \alpha:M_1 \to M_0 \text{ with } \hom(\Cone(\alpha),e[i]) = 0 ~\forall i.
\end{align*}

We note that for nonzero $\alpha \in \Hom(M_1,M_0)$, the
cone of $\alpha$ is a shifted semistable vector bundle $F$ of rank 3 and
degree zero: $F=\Cone(\alpha)[1]$.
Thus, as in Theorem \ref{cur-stab-con} we have an
orthogonal vector bundle $F$ to any semistable $e$.
This way we obtain a P-stability datum for rank $r$ bundles of degree
zero. Note that P-equivalence corresponds to S-equivalence of semistable
vector bundles (see also Tu's article \cite{Tu}).
This allows a new proof of Atiyah's classification of vector
bundles using the Fourier-Mukai transform $\FM_\kp$. For more details
see \S 14 in Polishuk's book \cite{Pol} and \cite{HP}.


\section{Example: surfaces}

For smooth, projective surfaces, we give a comparison theorem between
$\mu$-semistability and P-stability. A similar result is expected to
hold in any dimension. We assume $\chara(k)=0$ in this section, as we
make use of Bogomolov's restriction theorem.

\begin{localtheorem} \label{surface}
Let $X$ be a smooth projective surface and $H$ a very ample divisor on $X$. 
Given a Hilbert polynomial $p$, there is a P-stability datum $(C_\bullet, N)$
such that for any object $E\in\Db(X)$ the following conditions are equivalent:
\begin{abcliste}
\item[(i)]  $E$ is a $\mu$-semistable vector bundle with respect to $H$
             of Hilbert polynomial $p$
\item[(ii)] $E$ is P-stable with respect to $(C_\bullet, N)$.
\end{abcliste}
\end{localtheorem}

\begin{proof}
Suppose that $E$ is a $\mu$-semistable vector bundle with given Hilbert
polynomial $p$. As semistability implied that $E$ appears in a bounded 
family, there is an integer $m_0$ (depending only on $p$) such that $E(m_0)$ 
is $-2$-regular (in the sense of Mumford, see \cite{Mum-CuOnSf}). In 
particular, we have 
\begin{align*} 
 & H^i(E(m_0+k))=0 \text{ for all } i>0, k\geq -2, \\
 & E(m_0+k) \text{ is globally generated for } k\geq -2.
\end{align*}
Passing to the twist by $\ko_X(m_0)$, we may assume that $E$ itself is 
$-2$-regular. 

\proofpara{Sheaf conditions.}
By Theorem \ref{sc}, there are sheaves $C_{-1}$, $C_{-2}$, $C_{-3}$
and integers $N^j_i:=\hom^j(C_i,E)$ (for $i=-1,-2,-3$) such that any complex
$a\in\Db(X)$ with $\hom^j(C_i,a)=N^j_i$ is actually a $-2$-regular sheaf 
with Hilbert polynomial $p$. Thus, the first part of the P-datum consists
of these three objects $C_{-1}$, $C_{-2}$, $C_{-3}$.

\proofpara{Torsion freeness.}
To avoid torsion in $a$ we will construct sheaves $C_{-4}$ and $C_{-5}$ and 
add conditions of type $\hom^j(C_i,a)=N^j_i$ to the P-datum. We need two 
facts:

\begin{abcliste}
\item[(1)]
For a semistable vector bundle $E$ of given numerical invariants there
is an integer $m_1$ such that $H^0(E(k))=H^1(E(k))=0$ for all $k\leq m_1$.

\item[(2)]
If $a$ is a sheaf on $X$ with $H^0(a(k))=H^0(a(k-1))=H^1(a(k))=H^1(a(k-1))=0$
for some $k$, then $H^0(a(l))=H^1(a(l))=0$ for all $l\leq k$. (This assumes
that $\ko_X(1)=\ko_X(H)$ is very ample and that $H$ is general for $a$, i.e.\
does not contain the associated points of the sheaf $a$.)
\end{abcliste}

\noindent
Proof of (1): Along with $E$, the bundle $E\dual\otimes\omega_X$ is semistable
with certain prescribed numerics. Hence there exists $m'_1$ with 
$H^*(E\dual\otimes\omega_X(k'))=0$ for $k'\geq m'_1$. Then the statement follows
from Serre duality.

\noindent
Proof of (2): We use the hyperplane section sequence 
$0\to\ko_X(-H)\to\ko_X\to\ko_H\to0$. Tensorising this with $a$ and using 
dimensional induction shows the claim.

\noindent
Using the constant $m_1$ from (1) above, we require 
\begin{align*}
& h^i(a(m_1-j))   = h^i(E(m_1-j)) = 0,   & & j=0,1 \text{ and } i=0,1 \\
& h^2(a(m_1-j))   = p(m_1-j),            & & j=0,1.
\end{align*}
The second fact then implies $H^i(a(l))=0$ for all $l\leq m_1$ and $i=0,1$.
This in turn forces $a$ to be torsion free. If not, consider the torsion
exact sequence $0\to T\to a\to a/T\to 0$. Then $h^0(a(l))=0$ implies 
$h^0(T(l))=0$, hence $T$ is purely 1-dimensional. Next we have
$H^1(T(l))\cong H^0(a/T(l))$ for all $l\leq m_1$. Together with
$H^0(T(l))=0$, this shows that the polynomial $h^0(a/T(-l))$ is eventually
monotonously increasing. This is absurd for a coherent sheaf on a
projective variety --- contradiction.

\proofpara{Local freeness.} The sheaf $a$ is now automatically a vector
bundle. If not, we can consider the short exact sequence 
$0\to a\to a\ddual\to Q\to0$. As we know $a$ to be torsion free, $a\ddual$
is locally free and $Q$ a sheaf of dimension 0. Hence $h^0(Q(l))=\length(Q)$
for all $l$. But then $h^1(a\ddual(l))=\length(Q)$ for all $l\leq m_1$,
hence $\length(Q)=0$ by construction. 

Thus setting $C_{-4}:=\ko_X(1-m_1)$ and $C_{-5}:=\ko_X(-m_1)$ we can force
our objects to be vector bundles with the given numeric invariants.

\proofpara{Semistability.}
Using Bogomolov's restriction theorem (see \cite[\S7.3]{hl}), there is
a constant $m_2$ (again depending only on the numerics of $E$) such
that for all smooth curves $\tilde H\in|m_2H|$ the restriction
$E|_{\tilde H}$ is semistable. By the results about stability on 
curves as in Theorem \ref{cur-stab-con}, there is a bundle 
$F\in\Coh(\tilde H)$ with $\Hom(E|_{\tilde H},F)=\Ext^1(E|_{\tilde H},F)=0$,
i.e.\ $F\in(E|_{\tilde H})\orth$ in $\Db(\tilde H)$. By symmetry, $F$ is
also semistable on $\tilde H$. Hence there is a short exact sequence
of sheaves on $X$
\[ \xymatrix@1{
 0 \ar[r] & M \ar[r]^-\alpha & \ko_{\tilde H}^{\oplus r^2+1}(-m_3) \ar[r] 
          & F \ar[r] & 0 } \]
i.e.\ $F=\Cone(\alpha)$ is a torsion sheaf on $X$. 

On the other hand, if $F\in(E|_{\tilde H})\orth$ in $\Db(X)$, then 
$E|_{\tilde H}$ is semistable, hence $E$ is semistable with respect to
the polarisation $\tilde H$. Thus, once we have a map 
$\alpha\in\Hom(L,\ko_{\tilde H}^{\oplus r^2+1}(-m_3))$ with
$a\in\Cone(\alpha)\orth$ the sheaf $a$ will be semistable.
By setting $C_1:=M$ and $C_0:=\ko_{\tilde H}^{\oplus r^2+1}(-m_3)$
we thus complete our P-datum $(C_{-5},\dots,C_1,N)$.
\end{proof}

\begin{localproposition}
For every $h_1 \in H^2(X,\IZ)$ and each natural number $n\in\IN$ there
exists a Postnikov-datum $(C_i,N_i^j)$ such that for any
$a \in \Db(X)$ we have the equivalence
\[ \left( \begin{array}{c} \hom^j(C_i,a)=N^j_i \\
   \text{for all } j\in\IZ \text{ and all }i \end{array}\right)
 \iff
   \left( \begin{array}{c} a \cong A \in \Coh(X) \mbox{ with }
   A \cong L\otimes \kj_Z \\
   \mbox{where } L \mbox{ is a line bundle with } c_1(L)=h_1\\
   \mbox{and } \kj_Z \mbox{ is an ideal sheaf of colength } n.
\end{array} \right)
\]
\end{localproposition}

\begin{proof}
This follows the line of the proof of Theorem \ref{surface}.
Thus, we may assume the P-datum forces $a$ to be isomorphic to a
sheaf $A$ with the same
numerical invariants as $L \otimes \kj_Z$. We want to add conditions
to our P-datum which imply that $A$ is torsion free of the stated type.
We take an integer $m$ such
that for all $k \geq m$ and all line bundles $L$ with $c_1(L)=h_1$ we have
$h^0(L(k))=0=h^1(L(k))$. If we have $h^0(A(m-n-1))=0=h^0(A(m))$ and
$h^1(A(m-n-1))=n=h^1(A(m))$, then it immediately follows that the
torsion subsheaf $T \subset A$ is purely one dimensional. Let $\tilde H$
be a smooth divisor in $|(n+1)H|$ such that
\[ 0 \to T \to A \to A/T \to 0 \] remains exact when restricted to
$\tilde H$. The sheaf $T(m) \otimes \ko_{\tilde H}$ is of finite
length $l = c_1(T).\tilde H = (n+1) c_1(T).H$. Consequently, we have 
$T \ne 0$ implies that $l > n$. From the long exact cohomology sequence
\[ 0 = H^0(A(m)) \to H^0(A(m) \otimes \ko_{\tilde H}) 
                  \to H^1(A(m-n-1)) \cong k^n \]
we deduce $h^0(A(m) \otimes \ko_{\tilde H}) \leq n$.
Since $H^0(T(m) \otimes \ko_{\tilde H}) \subset H^0(A(m) \otimes
\ko_{\tilde H})$ we obtain $T=0$. Thus, adding the above conditions
forces $A$ to be torsion free.
\end{proof}


\section{Derived compactification of moduli spaces} \label{compactify}
Let $X=\IP^1 \times C$ be the product of $\IP^1$
with an elliptic curve $C$ with morphisms
\[\xymatrix{ \IP^1 & X \ar[r]^-q \ar[l]_-p & C\,.}\]
We denote a fiber of $p$ and $q$ by $f_p$ and $f_q$ respectively.
Denoting the class of a point by $z$ we have generators for the even
cohomology classes
\[\begin{array}{rclcl}
 H^0(X,\IZ) &=&\IZ  &=& \IZ\langle [X]\rangle \\
 H^2(X,\IZ) &=&\IZ^2&=& \IZ\langle f_q,f_p\rangle \\
 H^4(X,\IZ) &=&\IZ  &=& \IZ\langle z\rangle \,. 
\end{array} \]
We take the polarization $H=f_q+3f_p$ on $X$.

\paragraph{The moduli space $\km_1$.}
Let $E$ be a coherent sheaf of Chern character $\ch(E)=1+2f_q-2z$.
If $E$ is torsion free, then we have an isomorphism
$E \cong q^*L \otimes \kj_Z$ where $L$ is a line bundle of degree 2 on
$C$, and $Z$ is a closed subscheme of length 2.
We obtain by straightforward computations
\[  \Hom(E,E) = \IC, \qquad \Ext^1(E,E)=\IC^5, \qquad
    \Ext^2(E,E)=0,   \qquad \chi(E(k))= k^2+7k \,.\]
Thus, the moduli space $\km_1$ of torsion free coherent sheaves of Chern
character $\ch(E)=1+2f_q-2z$ is a smooth projective variety of dimension
5. Indeed, we have
\[ \Pic^2(C) \times \Hilb^2(X) \isom \km_1, \quad
        (L,Z) \mapsto q^*L\otimes\kj_Z . \]

\paragraph{The relative Fourier-Mukai transform.}
Choosing a base point $c \in C$ we can identify $C$ with its Picard
scheme $\Pic^0(C)$ and, as in Section \ref{conservation} obtain a
Poincare line bundle
$\kp$ on $C \times C$
subject to the conditions $\kp|_{\{c\} \times C} \cong \ko_C$ and
$\kp|_{C \times \{c\}} \cong \ko_C$.
From the diagram
$\xymatrix@1{C & C \times C \ar[l]_-{\pi_1} \ar[r]^-{\pi_2} &C}$
\mbox{ and the Fourier-Mukai transform }
 $\FM_\kp:\Db(C)\isom\Db(C)$ with
 $\FM_\kp(a) = \pi_{2*}(\kp \otimes \pi_1^*a)$
we obtain the diagram
 $\xymatrix@1{
   X & \IP^1 \times C \times C \ar[l]_-{\pi_{12}} \ar[r]^-{\pi_{13}}
  &X }$,
the line bundle $\kp_X=\pi_{23}^*\kp$ on $\IP^1 \times C \times C$,
and the Fourier-Mukai transform $\FM_{\kp_X}:\Db(X)\isom\Db(X)$ which is
defined by
$\FM_{\kp_X}(a):= \pi_{13*}(\kp_X \otimes \pi_{12}	^*a)$.

Next we study $\FM_{\kp_X}$ on objects parametrised by our moduli space
$\km_1$.
We remark that any object $E$ parametrised by $\km_1$ is the kernel of a
surjection $q^*L \to \ko_Z$ where $L$ is a degree two line bundle on $C$
and $Z$ is a length two subscheme of $X$.
The Fourier-Mukai transform $\FM_\kp(L)$ is a stable vector bundle of
rank two and degree $-1$ on the curve $C$ (see \cite[Chapter 14]{Pol}).
Thus, $E_L:=\FM_{\kp_X}(q^*L)=q^*\FM_\kp(L)$ is the pullback of a vector
bundle from $C$. 
Suppose a line bundle $M \cong \ko_X(n_pf_p + n_qf_q)$ is contained in $E_L$.
Since $E_L$ is trivial on the fibers of $q$ we find $n_p\leq0$.
The stability of $\FM_\kp(L)$ yields $n_q\leq -1$.
Thus, we have $c_1(M).H \leq -3$.

The Fourier-Mukai transform $\FM_{\kp_X}(\ko_Z)$ of the torsion sheaf
$\ko_Z$ is a sheaf with graded object $T_1 \oplus T_2$ where the sheaves
$T_i$ are line bundles of degree zero on fibers of $p$ extended to $X$.
We distinguish two cases:

Case 1: $\ko_Z$ is not contained in a fiber of $p$.
Since any morphism of $\FM_\kp(L)$ to a line bundle of degree zero is
surjective, applying $\FM_{\kp_X}$ to the short exact sequence 
\[ 0 \to E \to q^*L \to \ko_Z \to 0 \]
remains a short exact sequence of sheaves. Thus $\FM_{\kp_X}(E)$ is the
kernel of the surjective sheaf morphism $E_L \to \FM_{\kp_X}(\ko_Z)$.
Therefore $\FM_{\kp_X}(E)$ is a rank two vector bundle with
$c_1(\FM_{\kp_X}(E))=-f_q-2f_p$, and $c_2(\FM_{\kp_X}(E))=2z$.
Consequently, $c_1(\FM_{\kp_X}(E)).H =
-5$. Since any line bundle $M$ which is contained in $\FM_{\kp_X}(E)$ is
contained in $E_L$ this yields together with
\[ \frac{c_1(\FM_{\kp_X}(E)).H}{2} = \frac{-5}{2} > -3 \geq c_1(M).H \]
the stability of $\FM_{\kp_X}(E)$.

Case 2: $\ko_Z$ is contained in a fiber $p\inv(x)$. In this case the
morphism $E_L \to \FM_{\kp_X}(\ko_Z)$ cannot be surjective. Its
cokernel is a sheaf of length one concentrated on a point of $p\inv(x)$.
The kernel is a rank two vector bundle
$h^0(\FM_{\kp_X}(E))$ with numerical invariants
$c_1(h^0(\FM_{\kp_X}(E)))=-f_q-2f_p$, and $c_2(h^0(\FM_{\kp_X}(E)))=z$.

\paragraph{The moduli space $\km_2$.}
The second moduli space on $X$ we want to consider is the moduli space
$\km_2:=\km_{X,H}(2,-f_q-2f_p,3)$
of stable vector bundles $F$ on $X$ with
\[ \rk(F)=2, \qquad c_1(F)=-f_q-2f_p, \qquad c_2(F)=2z \,.\]
Suppose now that $[F] \in \km_2$. We consider the Fourier-Mukai transform
$\FM_{\kp_X}(F)$.
By construction, the cohomology of $\FM_{\kp_X}(F)$ lives only in degrees
zero and one.
\begin{locallemma}\label{exa-sheaf}
The Fourier-Mukai transform $\FM_{\kp_X}(F)$ has only first cohomology.
This means $\FM_{\kp_X}(F)=h^1(\FM_{\kp_X}(F))[-1]$.
\end{locallemma}
\begin{proof}
Suppose that $\FM_{\kp_X}(F)$ has cohomology in degree zero,
that is $h^0(\FM_{\kp_X}(F)) \ne 0$. This implies that
$\Hom_{\Db(X)}(\ko_X(-mH),\FM_{\kp_X}(F)) \ne 0$ for $m \gg 0$.
Therefore, we obtain
\[\Hom_{\Db(X)}(\FM_{\kp_X}(\ko_X(-mH)),\FM_{\kp_X}\FM_{\kp_X}(F)) \ne 0 \,. \]
We write $\iota:X\isom X$ for the involution coming from the inversion
in the group law of $C$. Since $\FM_{\kp_X}\FM_{\kp_X}(F) = \iota^*F[-1]$,
and the cohomology of
$\FM_{\kp_X}(\ko_X(-mH))$ is exclusively in degree one, we obtain a nontrivial
homomorphism $\psi:\iota^*h^1(\FM_\kp(\ko_X(-mH))) \to F$.
The restriction of $\iota^*h^1(\FM_\kp(\ko_X(-mH)))$ to any fiber of $p$ is
a stable vector bundle of rank $m$ and degree one. Thus the restriction
of $\psi$ to any fiber is not surjective. Therefore, the image of $\psi$
is of rank one. Let $L$ be the saturation of $\im(\psi)$ in $F$. We have
$L=\ko_X(n_pf_p+n_qf_q)$. Since $L$ contains the image of $\psi$ we have 
$n_q\geq1$. The stability of $F$ yields $n_p\leq -3-3n_q$. We have a short
exact sequence
\[0 \to L \to F \to \det(F) \otimes L\inv \otimes \kj_Z \to 0 \]
where $\kj_Z$ is the ideal sheaf of a subscheme $Z$ of finite length
\[ \length(Z) = c_2(F) - c_1(L).c_1(\det(F) \otimes L\inv) 
              = 2 + 2n_q+n_p+2n_p n_q \,. \]
However, the inequalities for $n_p$ and $n_q$ force $\length(Z)$ to be negative
which is impossible.
\end{proof}

\begin{locallemma}
The sheaf $h^1(\FM_{\kp_X}(F))$ is torsion free.
\end{locallemma}
\begin{proof}
We consider the torsion subsheaf $T(h^1(\FM_{\kp_X}(F)))$. It
contains a subsheaf $T$ which is of rank one on its support and
with $\supp(T)$ is irreducible. If the support of $T$ is zero-dimensional, 
then the morphism
$T[-1] \to \FM_{\kp_X}(F)$ defines via the Fourier-Mukai transform a
morphism from the torsion sheaf $\iota^*\FM_{\kp_X}(T)$ to $F$ which is
impossible because $F$ is torsion free. 

Thus, we may assume that $Y=\supp(T)$ is of dimension one. If the
induced morphism $p|_Y:Y \to \IP^1$
dominates $\IP^1$, then the the restriction of $\FM_{\kp_X}(T)$ to the
fibers of $p$ is a semistable vector bundle of degree zero. Hence we
obtain a morphism $\iota^*\FM_{\kp_X}(T) \to F$ which has image of rank
one and conclude like in the proof of Lemma \ref{exa-sheaf}.

Thus, we have to discuss only the case when $Y$ is a fiber of $p$, and
may assume that $T$ is torsion free on $Y$. We distinguish three cases
depending on the degree $\deg_Y(T)$ of $T$ on $Y$. In all these cases we
investigate the resulting morphism $\iota^*\FM_{\kp_X}(T) \to F$.

Case $\deg_Y(T) > 0$: Here $\iota^*\FM_{\kp_X}(T)$ is a torsion sheaf on
$Y$ which is impossible because $F$ is torsion free.

Case $\deg_Y(T) =0$: Here $\iota^*\FM_{\kp_X}(T) = k(y)[-1]$ for a point
$y \in Y$. However, for a vector bundle $F$ we have
$\Hom_{\Db(X)}(k(y)[-1],F) = \Ext^1(k(y),F)=H^0(\shExt^1(k(y),F))=0$.

Case $\deg_Y(T) < 0$: The Fourier-Mukai transform $\iota^*\FM_{\kp_X}(T)$
is a stable vector bundle $E_T$ on $Y$ of degree one shifted by $[-1]$.
The existence of a nontrivial homomorphism $\iota^*\FM_{\kp_X}(T) \to F$ is
equivalent to $\Hom(E_T,F|_Y) \ne 0$. 
The image of the morphism $E_T \to F|_Y$ is a line bundle on $Y$ of 
positive degree. Thus, we have
a surjection $F|_Y \to L$ where $L$ is a line bundle on $Y$ of degree $d
\leq -2$.
Denoting the kernel of the composition morphism
$F \to F|_Y \to L$ by $F'$, we obtain a vector bundle with invariants
\[ \rk(F')=2, \quad c_1(F')=-f_q-3f_p, \quad c_2(F')=3+d, \quad
   \Delta(F')=c_1^2(F')-4c_2(F') = -4d-6\,.\]
By Bogomolov's inequality \cite[Thm.\ 3.4.1]{hl}, $F'$ is 
Bogomolov unstable.
Hence there exists a destabilising exact sequence
\[0 \to M \to F' \to \det(F')\otimes M\inv \otimes \kj_Z \to 0\]
with $Z$ of dimension zero, $(2c_1(M)-\det(F')).H >0$, and
$(2c_1(M)-\det(F'))^2>0$.
Writing $c_1(M)=n_pf_p+n_qf_q$, we obtain the two inequalities
$2n_p+6n_q+6 >0$, and $(2n_p+3)(2n_q+1)>0$.
The first implies that not both factors in the second can be negative.
Thus, $n_p \geq -1$ and $n_q \geq 0$. Now $M$ is a subsheaf of $F$, too.
Thus, we have $(2c_1(M)-c_1(F)).H \leq 0$ which reads 
$2n_p+3n_q+5 \leq 0$, and is impossible.
\end{proof}

Putting together our results we have obtained the

\begin{localcorollary}
The Fourier-Mukai transform identifies the open subset $U$ of $\km_1$
which parametrises twisted ideal sheaves of two points in different fibers of
$p:X \to \IP^1$ with the moduli space $\km_2$.
\end{localcorollary}

\noindent
{\bf Two compactifications of the moduli space $\km_2$.}
At this point it seems natural to compactify $\km_2$ by adding the objects
$\FM_{\kp_X}(E)$ with $[E] \in \km_1 \setminus U$. Thus way we have an
isomorphism
$\FM_{\kp_X}: \km_1 \to \overline{\km}_2$. Since the dimensions of
$\Ext^i(a,a)$ are invariant under $\FM_{\kp_X}$, we obtain a smooth
moduli space $\overline{\km}_2$ from the smooth moduli space $\km_1$. We are
compactifying with simple objects having two cohomology sheaves. 

The classical construction of moduli spaces compactifies $\km_2$ with
coherent sheaves $E$ with one singular point. That is, the morphism $E
\to E^{\lor \lor}$ has cokernel of length one.
The Fourier-Mukai transform $\FM_{\kp_X}$
of these objects does not yield torsion free sheaves. This can be seen
be applying $\FM_{\kp_X}$ to the distinguished triangle containing the
morphism $E  \to E^{\lor \lor}$.
We hope this illustrates that the compactification $\overline{\km}_2$
by derived objects is natural and important.

\section{Sheaf conditions for objects in $\Db(X)$} \label{sheafconditions}
{\bf Notation:}
Let $X$ be a projective variety of dimension $n$
with a very ample polarisation $\ko_X(1)$.
For an object $a \in \Db(X)$, we denote the $i$th cohomology of the
complex $a$ by $a^i:=h^i(a)$. The object $a$ can be represented by a
sheaf concentrated in zero if and only if
$a^i=0$ for all integers $i \ne 0$.
Abbreviating, we call such an object $a$ a sheaf in $\Db(X)$.
For an object $a \in \Db(X)$,
we define the cohomology group $H^i(a(k))$ to be the vector space
$\Hom_{\Db(X)}(\ko_X(-k)[-i],a)$.

To compute the cohomology groups $H^i(a)$, we use the 
(local $\Rightarrow$ global) cohomology spectral sequence
of \cite[p.\ 263]{GM}:
\[ E^{pq}_2 = \Ext^q(\ko_X(-k),a^{-p}) =H^q(a^{-p}(k))
\Rightarrow H^{p+q}(a(k)) \,. \]

If the dimension of $X$ is zero this spectral sequence degenerates and
we conclude the

\begin{locallemma}\label{sc-0}
  If $X$ is of dimension zero, then $a \in \Db(X)$ is a sheaf if and
only if $H^i(a)=0$ for all $i \ne 0$. \qed
\end{locallemma}

\begin{locallemma}\label{sc-1}
Let $X$ be of dimension one, $p: \IZ \to \IZ$ be a polynomial
 of degree less than two.
  If $a \in \Db(X)$ is an object satisfying
\[ \begin{array}{ll}
(i) & H^i(a(k))= 0  \mbox{ for all pairs } (i,k)
\mbox{ with } k\in \{ -1, 0, p(-1) \} \,, \mbox{ and } i\ne0\,,\\
(ii)& \dim(H^0(a(k)))=p(k) \mbox{ for  } k\in \{-1,0\}\,,
\end{array}\]
then $a$ is a sheaf of Hilbert polynomial $p$.  
\end{locallemma}
\proof
Put $m:=p(-1)$. By choosing a two-dimensional vector subspace 
$V\subset H^0(\ko_X(1))$ such that $V \otimes \ko_X \to \ko_X(1)$ is 
surjective, we get a morphism $V \otimes H^0(a(-1)) \to H^0(a)$. The 
object $S:=S^{m-1}(V,\ko_X,\ko_X(1))$ obtained in Construction \ref{Sm} 
has, by Lemma \ref{lemma-2}, the following property:
the resulting morphism $\varrho_v:H^0(a(-1)) \to H^0(a)$ is injective 
for a general $v \in V$ if and only if $\Hom(S,a[-1])=0$.
As noted in Remark \ref{Smlinebundle}, $S$ is a vector bundle of 
rank one and determinant $\ko_X(-m)$, i.e.\ $S\cong\ko_X(-m)$.
Thus, by assumption (i) we get 
 $0=\Hom(S,a[-1])=H\inv(a(m))$. 
Now consider a general divisor $H$ in the linear system $\IP(V\dual)$ 
giving a distinguished triangle
\[ a(-1) \rarpa{v} a \to a \otimes \ko_H \to a(-1)[1] \,.\]
By general divisor we mean: That the morphism $H^0(a(-1)) \to H^0(a)$ is
injective and that $- \otimes \ko_H$ commutes with the cohomology of
the complex, i.e. $h^i(a \otimes \ko_H)=h^i(a) \otimes \ko_H$.
We derive that $a \otimes \ko_H$ fulfills the assumption of Lemma
\ref{sc-0} and is a sheaf. Thus for $i \ne 0$ the cohomology sheaves
$a^i$ are skyscraper sheaves concentrated in points outside $H$. Hence,
the above spectral sequence fulfills $E^{pq}_2 =0$ unless $p$ or $q$ are
zero. Since $H^i(a)=0$ for all $i<0$, we have $a^i=0$ for all $i<0$.
We consider now the distinguished triangle corresponding to the natural
t-structure
\[ \tau_{\leq 0}a \to a \to \tau_{ \geq 1}a \to \tau_{\leq 0}a[1] \,.\]
We have that $\tau_{\leq 0}a$ is a sheaf, which implies
$H^i(\tau_{\leq 0}a)=0$ for
$i \not\in\{0,1\}$. By assumption (i) we have $H^i(a)=0$ for all
$i\ne 0$, and from the (local $\Rightarrow$ global) spectral
sequence we deduce
that $H^i(\tau_{ \geq 1}a)=0$ for $i \leq 0$,
because $h^i(\tau_{ \geq 1}a)=0$ for $i \leq 0$.
From the long exact
cohomology sequence of this triangle we conclude $H^i(\tau_{ \geq 1}a)=0$
for all integers $i$. Since $\tau_{ \geq 1}a$ is a zero-dimensional
complex this implies $\tau_{ \geq 1}a =0$.
\qed

\begin{locallemma}\label{sc-2}
Let $X$ be of dimension two,
$V \subset H^0(\ko_X(1))$ a subspace such that the evaluation morphism
$V \otimes \ko_X \to \ko_X(1)$ is surjective, and
$p: \IZ \to \IZ$ be a polynomial
of degree less than three. Its derivative is the polynomial
$p'(k):=p(k)-p(k-1)$. Put $m:=(\dim(V)-1)(p(0)-1)$.
Any object $a \in \Db(X)$ with
\[ \begin{array}{ll}
(i)    & H^i(a(k))= 0  \mbox{ for all pairs } (i,k) \mbox{ with } 
         k\in \{ -2,-1, 0 \} \,, \mbox{ and } i\ne0\,,\\
(ii)   & \dim(H^0(a(k)))=p(k) \mbox{ for  } k\in \{-2,-1,0\}\,,\\
(iii_1)& \Hom_{\Db(X)}(S^m(V,\ko_X,\ko_X(1)),a[i]) =0 
         \mbox{ for all } i\ne0\,,\\
(iii_2)& \Hom_{\Db(X)}(S^m(V,\ko_X,\ko_X(1))\otimes\ko_X(1),a[i]) =0 
         \mbox{ for all } i\ne0\,,\\
(iii_3)& \Hom_{\Db(X)}(S^m(V,\ko_X,\ko_X(1))\otimes\ko_X(-p'(-1)),a[i])=0
         \mbox{ for all } i\ne0\,,\\
(iii_4)& \Hom_{\Db(X)}(\ko_X(-p'(-1)),a[i])=0 \mbox{ for all } 
         i\ne0\,,\mbox{ and}\\
(iii_5)& \Hom_{\Db(X)}(\ko_X(1-p'(-1)),a[i]) =0 \mbox{ for all } i\ne0 \\
\end{array}\]
is a sheaf of Hilbert polynomial $p$.  
\end{locallemma}
\proof
For $v \in V$ we have short a exact sequence
 $0 \to \ko_X(-1) \rarpa{v} \ko_X \to \ko_{H_v} \to 0$. 
We obtain long exact cohomology sequences 
\[
 H\inv(a(k) \otimes \ko_{H_v}) \to H^0(a(k-1)) \larpa{H(v)}
 H^0(a(k)) \to H^0(a(k) \otimes \ko_{H_v}) \to H^1(a(k-1)) .
\]
We will restrict ourselves to those $v \in V$ which are general in the
sense that they commute with forming cohomology, that is
\[ h^i(a \otimes \ko_{H_v}) = h^i(a) \otimes \ko_{H_v} 
\qquad \mbox{and} \qquad a \otimes \ko_{H_v} =
a \stackrel{{\rm L}}{\otimes} \ko_{H_v} \,.\]
By assumption (i) and (iii$_4$) the left-most terms are zero for
$k \in \{-1,0,p'(-1) \}$. For these values of $k$ the right-most
terms vanish by assumptions (i) and (iii$_5$).
We deduce from Lemma \ref{lemma-2} and (iii$_3$) the injectivity of $H(v)$
for $k=0$ and $v \in V$ general. Analogously it follows from (iii$_4$)
and (iii$_5$) that for a general $v$ the morphism $H(v)$ is injective
for $k=-1$ and $k=p'(-1)$. Thus, for a general $v \in V$ the tensor
product $a \otimes \ko_{H_v}$ fulfills the requirements (i) and (ii)
of Lemma \ref{sc-1} with $p$ replaced by $p'$.
From now on we suppose that $v \in V$ is general in the above sense.
Therefore $a \otimes \ko_{H_v}$ is a sheaf concentrated on $H_v$.
As in the proof of Lemma \ref{sc-1}, we obtain that for $i\ne 0$
the cohomology sheaves $a^i$ have supports disjoint from the ample
divisor $H_v$. Thus, they are supported in closed points
disjoint from $H_v$.
As before, the (local
$\Rightarrow$ global) spectral sequence yields $a^i=0$ for all $i <0$.
By Mumford's regularity criterion (see chapter 14 in
\cite{Mum-CuOnSf}) $a^0 \otimes \ko_{H_v}$ is 0-regular
which gives the vanishing $H^i(a^0 \otimes \ko_{H_v}(k))$ for all $k
\geq -1$ and $i>0$.
Thus, from the long exact sequence we obtain $H^2(a^0(k-1)) \cong
H^2(a^0(k))$ for all $k \geq 0$. Therefore we conclude $H^2(a^0(k))=0$
for all $k \geq -1$.
As in the proof of Lemma \ref{sc-1} we consider the filtration triangle 
\[ \tau_{\leq 0}a \to a \to \tau_{ \geq 1}a \to \tau_{\leq 0}a[1] \,,\]
and remark that $\tau_{\leq 0}a = a^0$ is a sheaf.
Since $H^i(\tau_{\leq 0}a(k))=0$ for all $i >2$ we deduce from the
associated long exact sequence $H^1(\tau_{ \geq 1}a(k)) \cong
H^2(\tau_{\leq 0}a(k))$ and $H^i(\tau_{ \geq 1}a(k))=0$ for $i \ne 1$.
Since $H^2(\tau_{\leq 0}a(k))$ vanishes for $k \geq -1$ and
$H^1(\tau_{ \geq 1}a(k))$ is (not canonically) twist invariant,
we deduce that
$H^1(\tau_{ \geq 1}a(k))=0$. This implies $\tau_{ \geq 1}a=0$.
\qed

\begin{localtheorem}\label{sc}
Let $X$ be a projective variety of dimension $n\leq 2$
and $p: \IZ \to \IZ$ be a polynomial of degree at most $n$.
There exists a sheaf
$b \in \Db(X)$ such that any object $a \in \Db(X)$ with
\[ \begin{array}{ll}
(i) & H^i(a(k))= 0  \mbox{ for all pairs } (i,k)
\mbox{ with } k\in \{-n, \ldots, 0 \} \,, \mbox{ and } i\ne0\,,\\
(ii)& \dim(H^0(a(k)))=p(k) \mbox{ for  }
k\in \{-n,\ldots ,0\}\,, \mbox{
and}\\
(iii)&\Hom_{\Db(X)}(b,a[i]) =0  \mbox{ for all } i\ne0\\
\end{array}\]
is a sheaf of Hilbert polynomial $p$.
Furthermore, any sheaf $a\in\Db(X)$ satisfying
conditions (i) and (ii) fulfills (iii).
\end{localtheorem}
\proof For dimension of $X$ equal to zero we can set $b=0$ and are done
by Lemma \ref{sc-0}.

In case $\dim(X)=1$ we set $b=\ko_X(-p(-1))$.
Lemma \ref{sc-1} tells us that conditions (i)--(iii) force $a$ to be
a sheaf.  To see that a sheaf
$a \in \Db(X)$ which satisfies (i) also satisfies (iii), we
remark that (i) implies the 0-regularity of the sheaf $a$. Thus, 
condition (iii) holds because $p(-1)$ being the dimension of a vector
space can not be negative.

If $\dim(X)=2$, then we set $b=b_1 \oplus b_2 \oplus b_3 \oplus b_4
\oplus b_5$ with $b_i$ the sheaf of condition (iii$_i$) of Lemma
\ref{sc-2}. (For example: $b_2=S^m(V,\ko_X,\ko_X(1))\otimes\ko_X(1)$.)
Again Lemma \ref{sc-2} tells us that conditions (i)--(iii) for $a \in
\Db(X)$ imply that $a$ is a sheaf. Suppose now that $a$ is a sheaf
fulfilling (i) and (ii). Since for a general $v \in V$ the resulting
morphisms $H^0(v):H^0(a(k-1)) \to H^0(a)$ are injective we obtain by
Lemma \ref{lemma-2} that conditions (iii$_1$), (iii$_2$), and (iii$_3$)
of Lemma \ref{sc-2} hold. Again the the Mumford-Castelnuovo regularity
of $a$ yields that conditions (iii$_4$), and (iii$_5$) hold, too.
\qed

\begin{localremarks}
(1) Considered as an element of the Grothendieck group $K(X)$ the sheaf $b$
of the above theorem
is in the subgroup spanned by the elements $\ko_X(k)$ with $k=0,
\ldots,n$. Thus, for an object $a$ which fulfills the conditions of the
theorem the dimension of $\Hom_{\Db(X)}(b,a)$ is given.

(2) The object $b$ depends on the Hilbert polynomial $p$ of $a$.
This can be seen best in Lemma \ref{sc-1}.

(3) If $a \in \Db(X)$ is a sheaf with Hilbert polynomial $p$,
then conditions (i) and (ii) of Theorem \ref{sc} do in general
not hold. However, after  a suitable twist these conditions hold.
\end{localremarks}

\section{The Euler triangle} \label{lemmasection}

\begin{locallemma}\label{lemma-1}
Let $U$ and $W$ be $k$-vector spaces of finite dimension.
Suppose that the morphism $\xymatrix{ U \otimes \ko_{\IP^n}
\ar[r]^-\rho & W \otimes \ko_{\IP^n}(1)}$ on $\IP^n$
is not injective. 
Then for any integer
$m \geq (\dim(U)-1)n$ we have $H^0(\ker(\rho)(m)) \ne 0$.
\end{locallemma}
\proof
From the morphism $\rho$ we obtain two short exact sequences
\[0 \to \ker(\rho) \to U \otimes \ko_{\IP^n}
\to \im(\rho) \to 0,
\qquad
0 \to \im(\rho) \to W \otimes \ko_{\IP^n}(1) 
\to \coker(\rho) \to 0\,.
\]
The resulting long cohomology sequences yield two inequalities for
all integers $k$
\[\begin{array}{rcl}
h^0(\ker(\rho)(k)) & \geq &
h^0(U \otimes \ko_{\IP^n}(k)) -h^0(\im(\rho)(k))\\
h^0(\im(\rho)(k)) & \leq & h^0(W \otimes \ko_{\IP^n}(1)) \,.
\end{array}\]
First we assume that $\dim(W) \leq \dim(U)-1$. This implies 
$h^0(\im(\rho)(m)) \leq (\dim(U)-1)\binom{n+1+m}{n}$.
Since $h^0(U \otimes \ko_{\IP^n}(m))=\dim(U)\binom{n+m}{n}$,
this yields 
$h^0(U \otimes \ko_{\IP^n}(m)) > h^0(\im(\rho)(m)) $ for all $m \geq
(\dim(U)-1)n$. Thus, we obtain $h^0(\ker(\rho)(m)) >0$
for $m \geq (\dim(U)-1)n$.

Now we assume that $\dim(W) \geq \dim(U)$. The cokernel $\coker(\rho)$
has rank at least $\dim(W)-\dim(U)+1$. Therefore there exists a subspace 
$W' \subset W$ of dimension $\dim(W)-\dim(U)+1$ such that the resulting
morphism $W' \otimes \ko_{\IP^n}(1) \to \coker(\rho)$ is injective in the
generic point, and eventually injective. Thus, the image of
the injective morphism
$H^0(\im(\rho)(k)) \to H^0(W \otimes \ko_{\IP^n}(k+1))$ is
transversal to $H^0(W' \otimes \ko_{\IP^n}(k+1))$. This implies 
$h^0(\im(\rho)(m)) \leq (\dim(U)-1)\binom{n+1+m}{n}$ as before.
\qed

\begin{localconstruction}\label{Sm}
The Euler triangle and objects $S^m(V,a,b)$.
\end{localconstruction}
For any two objects $a,b$ of a $k$-linear triangulated category $\kt$ and some
subspace $V\subset\Hom(a,b)$ of finite dimension we get a distinguished (Euler)
triangle
\[ S^m(V,a,b) \to \Sym^{m+1}(V)\otimes a \map{\theta} \Sym^m(V)\otimes b 
              \to S^m(V,a,b)[1] \]
where tensor products of vector spaces and objects are just finite direct sums, 
and $\theta$ is induced by the natural map
\[ \Sym^{m+1}(V)\to\Sym^m(V)\otimes\Hom(a,b), \quad
   f_0\vee\dots\vee f_m \mapsto 
  \sum_i (f_0\vee\cdots\hat{f_i}\cdots\vee f_m)\otimes f_i . \]

\begin{localremark}
In the special case where $\kt=\Db(\IP^n_k)$ is the bounded derived
category of the projective space $\IP^n_k$ over $k$ and $a=\ko_{\IP^n}$,
$b=\ko_{\IP^n}(1)$, $V=\Hom(a,b)=H^0(\ko_{\IP^n}(1))$ and $m=0$, the above
triangle is induced by the classical Euler sequence
\[ 0 \to \Omega_{\IP^n}(1) \to \ko_{\IP^n}^{\oplus n+1} \to \ko_{\IP^n}(1) \to 0.\]
\end{localremark}

\begin{localremark}
For any $c\in\kt$, the triangle defining $S^m(V,a,b)$ yields a long exact sequence 
\[\xymatrix@C=1.4em
{  \Hom^{k-1}(b,c)\otimes\Sym^m(V\dual)\ar[r]
 & \Hom^{k-1}(a,c)\otimes\Sym^{m+1}(V\dual) \ar[r]
 & \Hom^{k}(S^m(V,a,b),c) \ar[dll] 
\\
   \Hom^{k}(b,c)\otimes\Sym^m(V\dual)\ar[r]
 & \Hom^{k}(a,c)\otimes\Sym^{m+1}(V\dual) \ar[r]
 & \Hom^{k+1}(S^m(V,a,b),c).
} \]
\end{localremark}

\begin{locallemma}\label{lemma-2}
Let $\kt$ be a triangulated $k$-linear category with finite-dimensional Hom's,
$a,b,c\in\kt$ objects with $\Hom^{-1}(a,c)=0$ and let $V\subset\Hom(a,b)$ be 
a subspace. Then the following conditions are equivalent:
\begin{abcliste}
\item[(i)]  The morphism $\varrho_v:\Hom(b,c)\to\Hom(a,c)$ is injective for
            $v \in V$ general.
\item[(ii)] $\Hom^{0}(S^m(V,a,b),c)=0$ holds for some
            $m\geq(\dim(V)-1)(\hom(b,c)-1)$.
\end{abcliste}
\end{locallemma}
\proof
We consider the morphism $\Hom(b,c)\to V\dual\otimes\Hom(a,c)$.
Together with the natural surjection 
 $V\dual\otimes\ko_{\IP(V\dual)}\to\ko_{\IP(V\dual)}(1)$,
this gives a morphism
\[ \varrho: \Hom(b,c) \otimes \ko_{\IP(V\dual)} \to 
            \Hom(a,c) \otimes \ko_{\IP(V\dual)}(1) \quad
   \mbox{ on} \quad \IP(V\dual) \,.\]
The injectivity of $\varrho$ is equivalent to the injectivity of the
maps $\varrho_v:\Hom(b,c)\to\Hom(a,c)$ for generic (or just one)
$v \in V$.
By Lemma \ref{lemma-1} this is equivalent to the injectivity of
\[ H^0( \varrho \otimes \ko_{\IP(V\dual)}(m)) :
   H^0(\Hom(b,c) \otimes \ko_{\IP(V\dual)}(m)) \to
   H^0(\Hom(a,c) \otimes \ko_{\IP(V\dual)}(m+1)) \]
for $m = (\dim(V)-1)(\hom(b,c)-1)$.
Since $\Hom^{-1}(a,c)=0$, the long exact cohomology sequence of the
triangle from Construction \ref{Sm} gives that the kernel of 
$H^0(\varrho\otimes\ko_{\IP(V\dual)}(m))$ is $\Hom^{0}(S^m(V,a,b),c)$.
\qed

\begin{localcorollary} \label{Smlinebundle}
Suppose that $L$ is a base point free line bundle on a smooth projective
variety $X$ and $V\subset H^0(L)$ is a subspace such that
$V\otimes\ko_X\to L$ is surjective. Then the object $S^m(V,\ko_X,L)$ is
a vector bundle with invariants
\[ \rk(S^m(V,\ko_X,L))=\genfrac{(}{)}{0pt}{1}{m+\dim(V)-1}{m+1}, \qquad
   \det(S^m(V,\ko_X,L))=L^{-\otimes\binom{m+\dim(V)-1}{m}} .\]
\end{localcorollary}

\begin{localremarknumbered} \label{cor-2}
For two vector bundles $A$ and $B$ on a smooth projective curve $X$ such
that the canonical map $\Hom(A,B)\otimes A\to B$ is surjective, we set
\[ a=B\dual[1], \quad b=A\dual[1], \quad V=\Hom(a,b)=\Hom(A,B) .\]
Then $S^m(V,a,b)$ is a locally free sheaf concentrated in degree $1$,
because the maps $\Sym^{m+1}(V)\otimes B\dual\to\Sym^m(V)\otimes A\dual$
are surjective for all $m\geq0$.

For any sheaf $c=E\in\Coh(X)$, the condition $\Hom^{-1}(a,c)=0$ is
equivalent to $H^0(E\otimes B)=0$. Setting
 $F_{r,d}:=S^m(V,a,b)\otimes\omega_X[1]$ and using Serre duality, we see
 that statement (ii) in Lemma \ref{lemma-2} is equivalent to
 $\Hom(E,F_{r,d})=0$.
\end{localremarknumbered}

\begin{localremark}
Concerning the functoriality of the objects $S^m(V,a,b)$, let us first 
introduce the relevant category $\kc^m(\kt)$: its objects are triples
$(V,a,b)$ consisting of two objects $a,b\in\kt$ and a subspace
$V\subset\Hom(a,b)$. A morphism $(V,a,b)\to(V',a',b')$ in $\kc^m(\kt)$
is given by two maps
 $\alpha : \Sym^{m+1}(V)\otimes a \to \Sym^{m+1}(V')\otimes a'$ and
 $\beta : \Sym^m(V)\otimes b \to \Sym^m(V)\otimes b'$
such that the following diagram commutes:
\[ \xymatrix{
 \Sym^{m+1}(V)\otimes a \ar[r]^\theta \ar[d]_{\alpha} &
  \Sym^m(V)\otimes b \ar[d]^{\beta} \\
 \Sym^{m+1}(V')\otimes a' \ar[r]^{\theta'} & \Sym^m(V')\otimes b'
} \]
\noindent
This way, $\kc^m(\kt)$ is a $k$-linear category. With $S^m(V,a,b)=\Cone(\theta)$
and $S^m(V',a',b')=\Cone(\theta')$ in the above diagram, the usual nuisance of
non-functoriality of cones in triangulated categories prevents $S^m$
from being a functor $\kc^m(\kt)\to\kt$.

This problem is related to defining the spherical twist functors, and
we can follow the approach of Seidel and Thomas \cite{Seidel-Thomas} in the 
geometric case, $\kt=\Db(X)$. They use the fact that the homotopy category of 
bounded below complexes of (quasi-coherent) injectives with bounded coherent 
homology is equivalent to the bounded derived category of coherent sheaves. 
The $S^m$-construction works just as well for the homotopy category of 
injectives, and taking cones becomes functorial then, because the morphisms 
are genuine complex maps. See \cite[\S2a,b]{Seidel-Thomas} for details.

Thus, we obtain functors $S^m:\kc^m(\Db(X))\to\Db(X)$ for all $m\in\IN$. 
Note that even if we set the subspace $V$ to be the full homomorphism 
space and fix either $a$ or $b$, the resulting functor $\Db(X)\to\Db(X)$
cannot be triangulated, except in the case $m=0$.

\end{localremark}


\end{document}